\documentclass{amsart}

\usepackage{todonotes}

\usepackage{amssymb, enumerate}
\usepackage{amsrefs}

\theoremstyle{plain}
\newtheorem{theorem}{Theorem}
\newtheorem{lemma}[theorem]{Lemma}

\newtheorem{prop}[theorem]{Proposition}

\newtheorem{characterization}[theorem]{Characterization}

\theoremstyle{definition}
\newtheorem{definition}[theorem]{Definition}

\theoremstyle{remark}
\newtheorem{remark}[theorem]{Remark}
\newtheorem{question}{Question}

\newcommand{\lp}{\left(}

\newcommand{\rp}{\right)}
\newcommand{\e}{\epsilon}
\newcommand{\R}{\mathbb{R}}

\newcommand{\Z}{\mathbb{Z}}
\newcommand{\E}{\mathbb{E}}
\newcommand{\G}{\mathbb{G}}

\newcommand{\N}{\mathbb{N}}
\newcommand{\cl}{\overline}
\newcommand{\ti}{\textit}
\renewcommand{\H}{\mathcal H}

\DeclareMathOperator{\CAT}{\textup{\text{CAT}}}

\DeclareMathOperator{\Isom}{\textup{\text{Isom}}}

\numberwithin{equation}{section}
\numberwithin{theorem}{section}

\begin{document}

\title{A metric characterization of snowflakes of Euclidean spaces}
\author{Kyle Kinneberg and Enrico Le Donne}

\date{\today}

\begin{abstract} 
We give a metric
characterization of snowflakes of Euclidean spaces.
Namely,
a metric space  is isometric to  
$\R^n$ equipped with a distance $(d_{\rm E})^\epsilon$, for some 
$n\in \N_0$ and $\epsilon\in (0,1]$, where $d_{\rm E}$ is the Euclidean distance,
 if and only if it    is locally compact,
   $2$-point isometrically homogeneous, and 
   admits dilations of any factor.
    \end{abstract}
    \maketitle
    
\section{Introduction}

Carnot groups are rich mathematical objects that possess intriguing geometric, algebraic, and analytic structure. Somewhat surprising, then, is the fact that one can describe Carnot groups from a standard metric point of view. Following ideas of Montgomery, Zippin, Mitchell, and Berestovski{\u\i}, in \cite{LD13}, the second-named author gave a strictly metric characterization of \ti{sub-Finsler} Carnot groups as those metric spaces that are locally compact, geodesic, isometrically homogeneous, and self-similar.

\begin{characterization}
\label{LD}
A metric space $(X,d)$ is isometric to a sub-Finsler Carnot group if and only if it has the following properties.
\begin{itemize}
\item[(\ref{LD}.i)]  $X$ is locally compact.
\item[(\ref{LD}.ii)]  $X$ is geodesic.
\item[(\ref{LD}.iii)]  $\Isom(X)$ is transitive on $X$.
\item[(\ref{LD}.iv)]  There is $\lambda >1$ and a homeomorphism $f_{\lambda} \colon X \rightarrow X$ such that \\
$d(f_{\lambda}(x),f_{\lambda}(y)) = \lambda d(x,y)$ for all $x,y \in X$.
\end{itemize}
\end{characterization}

Characterization \ref{LD} was demonstrated in \cite{LD13}. An analogous statement can be found in Berestovski{\u\i}'s doctoral thesis and in \cite[Example 3.2]{Bere}.

Each of these four properties is necessary to guarantee that $X$ is a Carnot group. For example, the space $(\R^n, |\cdot|^{1/2})$ is locally compact, homogeneous, and admits dilations. In fact, if $\G$ is a Carnot group with metric $d$, and if $0<\e<1$, then its \ti{snowflake} $(\G,d^{\e})$ has properties (\ref{LD}.i), (\ref{LD}.iii), and (\ref{LD}.iv), but fails to be geodesic. Indeed, $(\G,d^{\e})$ admits no non-constant rectifiable paths.

The natural question arises: are snowflakes of Carnot groups the only metric spaces which satisfy (\ref{LD}.i), (\ref{LD}.iii), and (\ref{LD}.iv)? The answer is negative, as can be seen by the following example. On $\R^2$, define the metric 
$$d(x,x') = |x_1-x_1'| + |x_2-x_2'|^{1/2};$$
here, the dilation $f_{\lambda}$ is $f_{\lambda}(x) = (\lambda x_1, \lambda^2 x_2)$. Similar examples can be constructed on any Carnot group that decomposes as a nontrivial product of Carnot groups: simply snowflake each factor independently. Moreover, a metric space satisfying (\ref{LD}.i), (\ref{LD}.iii), and (\ref{LD}.iv) may fail to be connected (see Section \ref{ex:Cantor}). We therefore must modify our question as follows.

\begin{question}
If a connected and locally connected metric space satisfies (\ref{LD}.i), (\ref{LD}.iii), and (\ref{LD}.iv), is it necessarily isometric to a product of snowflakes of Carnot groups?
\end{question}

We remark that by a result of Gleason, Montgomery, and Zippin each metric space in Question 1 has the structure of a manifold; see, for example, the argument in \cite{LD13}. In this paper, we work toward an answer to the question above by giving a metric characterization of snowflakes of the Euclidean spaces $\mathbb E^n$.

\begin{theorem} \label{main}
A metric space $(X,d)$ is isometric to a snowflake of some Euclidean space
$\mathbb E^n$,
$n\in \N\cup\{0\}$,
 if and only if it has the following properties.
\begin{itemize}
\item[(\ref{main}.i)] $X$ is locally compact.
\item[(\ref{main}.ii)]  $X$  is $2$-point isometrically homogeneous: for all $x_1, x_2,y_1, y_2 \in X$ such that
$d(x_1, x_2 )=d(y_1,y_2)$, there exists $f \in \Isom(X)$ such that $f(x_i)=y_i$ for $i=1,2$.
\item[(\ref{main}.iii)] $X$ admits dilations of any factor: for all $\lambda >0$, there is a homeomorphism $f_{\lambda} \colon X \rightarrow X$ such that $d(f_{\lambda}(x),f_{\lambda}(y)) = \lambda d(x,y)$ for all $x,y \in X$.
\end{itemize}
\end{theorem}
 
Recognizing snowflakes of Euclidean spaces, or more generally, of Carnot groups, is important for many rigidity problems in hyperbolic geometry. In particular, metric characterizations of these spaces can give corresponding uniformization results in the coarse hyperbolic category (e.g., for $\CAT(-1)$-spaces or for Gromov hyperbolic metric spaces) by looking at the visual boundary of the hyperbolic object. More specifically, this boundary is equipped with a class of visual metrics, any two of which are ``snowflake-equivalent" to each other. The classical hyperbolic spaces---non-compact, rank one symmetric spaces---have Carnot groups as their visual boundaries.

Interesting questions also arise from dynamical considerations. For example, if a Gromov hyperbolic metric space, $Y$, admits a geometric group action (i.e., a properly discontinuous, cocompact action by isometries), then this action passes naturally to an action on the visual boundary, $\partial Y$, by homeomorphisms. In fact, the induced action will include ``quasi-dilations" of any factor, at least locally (cf. \cite[Lemma 3.1]{Kin14} for a particular manifestation of this principle). Furthermore, the diagonal action on $\partial Y \times \partial Y$ is topologically transitive (cf. \cite[Theorem 6.3.6]{Nic89} in the classical setting, or use the ``north-south" dynamics described in \cite[Section 4]{KB02}).

Of course, the hypotheses we impose in Theorem \ref{main} are much stronger than the properties that are generally found in the boundaries of hyperbolic metric spaces. The motivation, however, should be clear. Let us mention a specific problem from this setting that has connections to the methods we use here.

Let $M$ be a complete, simply connected, Riemannian manifold with sectional curvature $\leq -1$ and which admits a geometric group action. Let $X$ denote its visual boundary, equipped with the canonical visual metric. Is it true that if $X$ contains a non-constant curve of finite length (or even stronger, if any two points in $X$ can be joined by a curve of finite length), then $M$ is symmetric? For a partial positive result in this direction see \cite[Section 4]{Con03}, where it is shown that the conclusion holds if there are ``many" rectifiable curves joining any two points in $X$.

\subsection*{Acknowledgement}
This work was initiated when the first-named author visited the
University of Jyv\"askyl\"a. Both authors would like to thank the
University of Jyv\"askyl\"a for excellent working conditions and thank
John Mackay and Gareth Speight  for improving feedback.

\section{Preliminaries}

Let us begin by establishing some easy consequences of properties (\ref{main}.ii) and (\ref{main}.iii). In the first place, we have
\\
\begin{itemize}
\item[(\ref{main}.iv)]  $\Isom(X)$ is transitive on $X$.\\
\end{itemize} 
Indeed, it is enough to take $x_1= x_2$ and $y_1= y_2$ in (\ref{main}.ii). Additionally, we have
\\
\begin{itemize}
\item[(\ref{main}.v)]  $X$ admits dilations of any factor fixing any base point: for all $x_0\in X$ and all $\lambda >0$, there is a homeomorphism $f_\lambda \colon X \rightarrow X$ for which $f_\lambda (x_0)=x_0$ and  $d(f_\lambda(x),f_\lambda(y)) = \lambda d(x,y)$ for all $x,y \in X$.\\
\end{itemize}
Indeed, one can compose a dilation given by (\ref{main}.iii) with an isometry from (\ref{main}.iv) in order to fix $x_0$. It will also be useful to observe that properties (\ref{main}.ii) and (\ref{main}.iii) can be unified into a single assumption. Namely,
\\
\begin{itemize}
\item[(\ref{main}.vi)] If $x,y,x',y' \in X$, with $x \neq y$ and $x' \neq y'$, 
then there is a homeomorphism $f \colon X \rightarrow X$ with $f(x)=x'$ and $f(y)=y'$ 
such that $d(f(u),f(v)) = \lambda d(u,v)$ for all $u,v \in X$, where $\lambda = \frac{d(x',y')}{d(x,y)}$.\\
\end{itemize}
Indeed, by (\ref{main}.v) we have a homeomorphism $f_1$ (respectively, $f_2$) fixing $x$ (respectively, $x'$) and scaling the distance by $\lambda_1 = 1/d(x,y)$ (respectively, $\lambda_2 = 1/d(x',y')$). Let $z = f_1(y)$ and $z' = f_2(y')$ so that $d(x,z)=d(x',z')$. Then by (\ref{main}.ii), 
there is $f_3 \in \Isom(X)$ for which $f_3(x)=x'$ and $f_3(z) = z'$. The desired homeomorphism $f$ is then $f_2^{-1} \circ f_3 \circ f_1$.

Unless $X$ consists of a single point (or is the empty set), we claim that any metric space that admits dilations of any factor has Hausdorff dimension at least $1$. We remark, though, that the dilations $f_\lambda$ are not assumed to be continuous in $\lambda$. Hence, one cannot immediately deduce that the metric space contains a non-constant curve.

\begin{lemma}\label{HDim1}
A metric space containing at least two points and with property (\ref{main}.v) has infinite Hausdorff $1$-measure, and hence Hausdorff dimension at least $1$. 
\end{lemma}

\begin{proof}
Apart from the base point $x_0$ we assume that there exists another point $x_1\in X$. Say $\bar r:=d(x_0,x_1) >0$. 
By property (\ref{main}.v) the point $f_\lambda(x_1) $ has distance $\lambda \bar r$ from $x_0$. We deduce that for all $r>0$ the metric sphere
$$S(x_0, r):=\{x\in X \colon d(x_0,x)=r\}$$
is not empty. Thus, the $1$-Lipschitz map $\rho \colon \R \rightarrow [0,\infty)$ defined by $\rho(x) = d(x_0,x)$ is surjective. In particular, its image has infinite Hausdorff $1$-measure. As is well-known, $1$-Lipschitz functions cannot increase Hausdorff measure (e.g., \cite[Theorem 7.5]{Mat95}), so we conclude that $\H^1(X)=\infty$.
\end{proof}

\begin{remark}[Completeness] \label{completeness}
For later uses, we point out that a metric space with properties
(\ref{main}.i) and
(\ref{main}.iv) is necessarily complete.
Indeed, if the space is locally compact then one can consider  a point $\bar x\in X$ and take a small-enough radius $r >0$ so that the ball $B(\bar x, r)$ is precompact, i.e., its closure is compact.
By isometric homogeneity, any other point $x\in X$ has $B( x, r)$ precompact.
If, now, $(x_n)$ is a  Cauchy sequence in $X$, then eventually $d(x_n,x_m)<r$. Thus, the sequence is eventually in a compact set, so it has an accumulation point; being Cauchy, this implies that the sequence converges.

If, in addition, the metric space has property (\ref{main}.v), then a similar argument shows that every closed ball is compact, i.e., the space is \ti{proper}.
\end{remark}

If $(X,d)$ is a metric space and $0 < \e < 1$, the metric space $(X,d^{\e})$ is called a snowflake of $(X,d)$. This terminology is motivated by the von Koch snowflake, which is bi-Lipschitz equivalent to a snowflake of the segment $[0,1]$ when equipped with the metric induced from the plane. Similarly, we say that a metric space $(X,d)$ can be \ti{de-snowflaked} if there is $p > 1$ such that $(X,d^p)$ is a metric space. In this case, the original metric $d$ is a snowflake of $(X,d^p)$, using $\e = 1/p$.

In many important situations, it makes sense to study snowflakes of metric spaces up to bi-Lipschitz equivalence. For example, given any two metrics in the class of visual metrics on the boundary of a Gromov hyperbolic metric space, they are either bi-Lipschitz equivalent or one is a snowflake of the other, up to bi-Lipschitz equivalence. J.~Tyson and J.-M.~Wu \cite{TyWu05} have given a nice characterization of metric spaces that are bi-Lipschitz equivalent to a snowflake. One of the conditions they introduce is a metric notion of uniform non-convexity (which is motivated by the construction of the von Koch snowflake). 

If $(X,d)$ is a metric space, $x,y \in X$, and $0 < \lambda < 1$ and $\delta > 0$ are constants, let
$$L(x,y; \lambda, \delta) := \cl{B}(x,(\lambda+\delta)d(x,y)) \cap \cl{B}(y,(1-\lambda+\delta)d(x,y))$$
be the corresponding ``lens-shaped" set. We say that $(X,d)$ is \ti{uniformly non-convex} if there is $0 < \delta < 1/2$ such that, for each pair of points $x,y \in X$, there is a value $\lambda \in (\delta, 1-\delta)$ for which $L(x,y;\lambda,\delta) = \emptyset$. 

\begin{prop}[{\cite[Theorem 1.5]{TyWu05}} and ensuing discussion] \label{nonconvex}
If a metric space $(X,d)$ is uniformly non-convex, then it is bi-Lipschitz equivalent to a snowflake.
\end{prop}

In this paper, we are interested in precise snowflakes, mostly because the assumptions we put on $X$ are strong enough that bi-Lipschitz equivalence is not needed. We will, however, use Proposition \ref{nonconvex} in an important way.

We should also remark that \cite{TyWu05} contains much more than Proposition \ref{nonconvex} in the way of characterizing snowflakes. For starters, the converse statement holds, as long as $(X,d)$ has a bi-Lipschitz embedding into a uniformly convex Banach space. We will not need these results here.

\section{Proof of Theorem \ref{main}}

Our goal in this section is to prove Theorem \ref{main}. Not surprisingly, in view of Characterization \ref{LD} and the discussion following it, we will proceed by demonstrating that any metric space satisfying the hypotheses of the theorem is either geodesic or can be de-snowflaked to a geodesic space. The resulting geodesic metric space is then isometric to a Carnot group which is $2$-point isometrically homogeneous. Standard arguments show that the only such Carnot groups are Euclidean spaces. Thus, the ``snowflake alternative" is the central step in our approach. This alternative hinges on the following notion, which makes sense in any metric space.

\begin{definition}
For a metric space $(X,d)$, we say that $z \in X$ is a \ti{between-point} if there are $x,y \in X \backslash \{z\}$ for which $d(x,y) = d(x,z) + d(z,y)$.
\end{definition}

Note that if $X$ is a metric space containing a non-constant geodesic segment, then every point on this segment (excluding the endpoints) is a between-point. Thus, geodesic metric spaces (with at least two points) have many between-points. On the other hand, if $(X,d)$ is the snowflake of some metric on $X$, then $d(x,y) < d(x,z) + d(z,y)$ for all $z \neq x,y$, so $(X,d)$ has no between-points. The alternative we will use for de-snowflaking subsists in the following proposition.

\begin{prop} \label{prop}
Let $(X,d)$ be a metric space with properties (\ref{main}.i), (\ref{main}.ii), and (\ref{main}.iii). If $X$ has a between-point, then $X$ is geodesic. If $X$ has no between points, then there is $p>1$ for which $(X,d^p)$ is a geodesic metric space. In either case, there is $p \geq 1$ for which $(X,d^p)$ is a geodesic metric space with properties (\ref{main}.i), (\ref{main}.ii), and (\ref{main}.iii).
\end{prop}

To prove this proposition, we will need several lemmas. All of them address the case in which $X$ has no between-points.

\begin{lemma} \label{TyWu}
Let $(X,d)$ be a metric space with properties (\ref{main}.i), (\ref{main}.ii), and (\ref{main}.iii). If $X$ does not have a between-point, then there is $p > 1$ and $L>0$ for which
$$d(x_0,x_n)^p \leq L\sum_{i=1}^n d(x_i,x_{i-1})^p$$
for all finite chains of points $x_0,\ldots,x_n$ in $X$.
\end{lemma}

\begin{proof}
Without loss of generality, we may assume that $X$ contains at least two points. For $0 < \delta <1/2$, let
$$L(x,y;\delta) = \cl{B} \lp x, \lp \delta + \tfrac{1}{2} \rp d(x,y) \rp \cap \cl{B} \lp y, \lp \delta + \tfrac{1}{2} \rp d(x,y) \rp$$
be the lens-shaped set discussed in the previous section, corresponding to the value $\lambda = 1/2$. We claim that there is a uniform value $\delta$ such that $L(x,y;\delta) = \emptyset$ for each pair of distinct points $x,y \in X$. 

To see this, first choose points $a,b \in X$ with $d(a,b)=1$; such a pair exists by (\ref{main}.iii) and the assumption that $X$ has at least two points. Note that if $z_{\delta} \in L(a,b;\delta)$, then $d(a,z_{\delta}) \leq \frac{1}{2} + \delta$ and $d(z_{\delta},b) \leq \frac{1}{2} + \delta$, so 
$$d(a,z_{\delta}) + \rho(z_{\delta},b) \leq 1 + 2\delta = d(a,b) + 2\delta.$$ 
If such $z_{\delta}$ were to exist for each $0 < \delta <1/2$, then, as closed balls in $X$ are compact (cf. Remark \ref{completeness}), we could find an accumulation point $z \in X$ for which $\rho(a,z), \rho(z,b) \leq \frac{1}{2}$ and 
$$\rho(a,z) + \rho(z,b) \leq 1.$$
Clearly, such $z$ is a midpoint for $a$ and $b$, contradicting the assumption that $X$ has no between-points. Thus, there is $0<\delta<1/2$ for which $L(a,b;\delta) = \emptyset$.

We now use property (\ref{main}.vi) to show that $L(x,y;\delta) = \emptyset$ whenever $x,y \in X$ are distinct. Indeed, there is a homeomorphism $f \colon X \rightarrow X$ with $f(x)=a$ and $f(y)=b$, which scales distances by $d(a,b)/d(x,y)$. It is easy to see that any $z \in L(x,y;\delta)$ would give $f(z) \in L(a,b;\delta)$, contrary to the choice of $\delta$. This establishes our claim.

Thus, $(X,d)$ is uniformly non-convex, in the language of \cite{TyWu05}. Invoking Proposition \ref{nonconvex}, we can conclude that $(X,d)$ is bi-Lipschitz equivalent to a snowflake. In other words, there is $p > 1$ for which $d^p$ is bi-Lipschitz equivalent to a metric on $X$. This implies that there is $L > 0$ such that
$$d(x_0,x_n)^p \leq L\sum_{i=1}^n d(x_i,x_{i-1})^p$$
for all finite chains of points $x_0,\ldots,x_n$ in $X$.
\end{proof}

For the following lemma, we need to introduce some notation, namely, a type of gauge function. Let us assume that $X$ contains two points $a,b \in X$ with $d(a,b)=1$. Define the function
$$\phi(p) = \inf_{x \in X} \lp d(a,x)^p + d(x,b)^p \rp, \hspace{0.5cm} p \geq 1,$$
which is upper semi-continuous in $p$. Indeed, for each $\e > 0$, there is $x \in X$ for which
$$ \begin{aligned}
\phi(p) &\geq d(a,x)^p + d(x,b)^p - \e \\
&= \limsup_{n \rightarrow \infty} \lp d(a,x)^{p_n} + d(x,b)^{p_n} \rp -\e \\
&\geq \limsup_{n \rightarrow \infty} \phi(p_n) - \e,
\end{aligned} $$
so that $\phi(p) \geq \limsup_{n \rightarrow \infty} \phi(p_n)$.

\begin{remark} \label{cts}
When $X$ is a proper metric space (i.e., closed balls are compact), it is not difficult to see that $\phi$ is, in fact, continuous. To verify this, first observe that the infimum defining $\phi(p)$ is attained for each $p$. Now, if $(p_n)$ is a sequence converging to $p$, let $x_n \in X$ be points at which $\phi(p_n)$ is attained. Note that these points all lie in a fixed ball centered at $a$, so after passing to a subsequence, we may assume that $\lim_{n \rightarrow \infty} x_n = x$. This implies that the quantity $\phi(p_n) = d(a,x_n)^{p_n} + d(x_n,b)^{p_n}$ converges to $d(a,x)^p + d(x,b)^p$, which is an upper bound for $\phi(p)$. We therefore obtain $\phi(p) \leq \liminf_n \phi(p_n)$, which gives lower semi-continuity of $\phi$.
\end{remark}

Let us also observe that $\phi(p) \leq 1$ for all $p$, since one can take $x=a$ in the infimum. Moreover, $\phi(p)=1$ if and only if the distance function $d^p$ on $X$ satisfies the triangle inequality for points $a$, $b$, and $x$, where $x \in X$ is arbitrary. The next lemma considers the case that $\phi(p) < 1$.

\begin{lemma} \label{smallchain}
Suppose that $X$ satisfies (\ref{main}.i), (\ref{main}.ii), and (\ref{main}.iii). If $\phi(p) < 1$, then there are chains $a=x_0,x_1,\ldots,x_n=b$ in $X$ for which
$$\sum_{i=1}^n d(x_i,x_{i-1})^p$$
is arbitrarily small.
\end{lemma}

\begin{proof}
As $\phi(p) < 1$, we know that there is $z \in X$ for which
$$d(a,z)^p + d(z,b)^p < 1.$$
Consider the following inductive construction of finite chains of points between $a$ and $b$. The first chain is $x_0=a$, $x_1=z$, and $x_2=b$. Given a chain
$$a=x_0,x_1,\ldots, x_{2^k}=b,$$
for which consecutive points are distinct, we form a new chain 
$$a=y_0,y_1,\ldots,y_{2^{k+1}}=b$$
by setting 
$$y_i = 
\begin{cases}
x_{i/2} &\text{if } i \text{ is even} \\
f_i(z) &\text{if } i \text{ is odd}.
\end{cases}$$
Here, $f_i$ is a homeomorphism given by (\ref{main}.vi) for the pairs $(a,b)$ and $(y_{i-1},y_{i+1})$, with $f_i(a) = y_{i-1}$ and $f_i(b) = y_{i+1}$. It is important to note that $y_{i-1} \neq y_{i+1}$, as consecutive points in the original chain were distinct. Moreover, the construction guarantees that consecutive points in the new chain are also distinct. 

Now observe that for even $i$, we have
$$\begin{aligned}
d(y_i,y_{i+1})^p &+ d(y_{i+1},y_{i+2})^p = d(f_{i+1}(a),f_{i+1}(z))^p + d(f_{i+1}(z),f_{i+1}(b))^p \\
& = \lp \frac{d(y_{i},y_{i+2})}{d(a,b)} \cdot d(a,z)\rp^p + \lp \frac{d(y_{i},y_{i+2})}{d(a,b)} \cdot d(z,b)\rp^p \\
&= \lp d(a,z)^p + d(z,b)^p \rp \cdot d(x_{i/2},x_{i/2+1})^p.
\end{aligned}$$
As a result, we see that
$$\begin{aligned}
\sum_{i=1}^{2^{k+1}} d(y_i,y_{i-1})^p &= \sum_{i \text{ even}} \lp d(y_{i-1},y_{i-2})^p + d(y_i,y_{i-1})^p \rp \\
&= \lp d(a,z)^p + d(z,b)^p \rp \sum_{i=1}^{2^k} d(x_i,x_{i-1})^p.
\end{aligned}$$
By induction, we then obtain
$$\sum_{i=1}^{2^{k+1}} d(y_i,y_{i-1})^p =  \lp d(a,z)^p + d(z,b)^p \rp^{k+1}.$$
As $d(a,z)^p + d(z,b)^p < 1$, this can be made arbitrarily small by taking large $k$.
\end{proof}

The following lemma is a baby version of the statement in Proposition \ref{prop} regarding metric spaces with no between points. It will not be difficult to upgrade it to the full version.

\begin{lemma} \label{existsp}
Let $(X,d)$ be a metric space, having at least two points, that satisfies (\ref{main}.i), (\ref{main}.ii), and (\ref{main}.iii). If $X$ has no between-points, then there is $p>1$ for which $(X,d^p)$ is a metric space with at least one between point.
\end{lemma}

\begin{proof}
As $X$ has at least two points and admits dilations, there exist $a,b \in X$ with $d(a,b)=1$. Let $\phi(p)$ be the corresponding gauge function, defined for $p \geq 1$. We remarked earlier that for general metric spaces, $\phi(p) \leq 1$, with equality if and only if $d^p$ satisfies the triangle inequality for the points $a$, $b$, and $x$, with $x \in X$ arbitrary. Using the $2$-point homogeneity in (\ref{main}.vi), we can strengthen this statement: if $\phi(p)=1$, then $d^p$ is a metric on $X$. To see this, fix $y,z \in X$ distinct and take a homeomorphism $f \colon X \rightarrow X$ such that $f(a)=y$, $f(b)=z$, and 
$$d(f(u),f(v)) = \frac{d(y,z)}{d(a,b)} d(u,v) = d(y,z)d(u,v)$$
for all $u,v \in X$. Then, for any $x \in X$, we have
$$\begin{aligned}
d(y,z)^p &= d(y,z)^p \phi(p) \\
&\leq d(y,z)^p \lp d(a,x)^p + d(x,b)^p \rp \\
&= \lp d(y,z)d(a,x)\rp^p + \lp d(y,z)d(x,b) \rp^p \\
&= d(f(a),f(x))^p + d(f(x),f(b))^p \\
&= d(y,f(x))^p + d(f(x),z)^p.
\end{aligned}$$
As $f$ is surjective, this implies the triangle inequality for $(X,d^p)$. Verifying the other metric properties is, of course, straightforward.

Our next claim is that the set $\phi^{-1}(\{1\})$ is non-empty, closed, and bounded in $[1,\infty)$. Non-emptiness comes from the obvious fact that $\phi(1)=1$. The continuity of $\phi$, which is guaranteed by Remark \ref{cts} and the fact that $X$ is proper (cf. Remark \ref{completeness}), implies that this set is closed. To show boundedness, we argue as follows. Properties (\ref{main}.i), (\ref{main}.ii), and (\ref{main}.iii) guarantee that $(X,d)$ is metrically doubling: every ball of radius $r>0$ can be covered by at most $C$ balls of radius $r/2$, where $C$ is a uniform constant. In particular, this means that $(X,d)$ has finite Hausdorff dimension, say equal to $D$. If $\phi(p)=1$, then $(X,d^p)$ is a metric space with Hausdorff dimension $D/p$. Moreover, it satisfies the requirements of Lemma \ref{HDim1}, so $D/p \geq 1$; equivalently, $p \leq D$.

Let $p\geq 1$ be the maximal value for which $\phi(p)=1$. We now claim that the metric space $(X,d^p)$ has a between-point. First observe that $(X,d^p)$ satisfies properties (\ref{main}.i), (\ref{main}.ii), and (\ref{main}.iii). If it does not have a between-point, then Lemma \ref{TyWu} guarantees that there is $q > 1$ and $L > 0$ such that 
$$d(x_0,x_n)^{pq} \leq L\sum_{i=1}^n d(x_i,x_{i-1})^{pq}$$
for all finite chains of points $x_0,\ldots,x_n$ in $X$. On the other hand, maximality of $p$ ensures that $\phi(pq) <1$. By Lemma \ref{smallchain}, there are chains $a=x_0,x_1,\ldots,x_n=b$ for which
$$\sum_{i=1}^n d(x_i,x_{i-1})^{pq}$$
is arbitrarily small. Together, these facts contradict $d(a,b)^{pq} = 1$. Thus, $(X,d^p)$ must have at least one between-point, as desired.
\end{proof}

With these lemmas, we are ready to prove Proposition \ref{prop}. There are essentially two statements to verify, corresponding to whether or not $(X,d)$ has between-points. We treat the case of existence of between-points first, as we will use this result to establish the second case.

\begin{proof}[\textbf{Proof of Proposition \ref{prop}}]
Suppose that $(X,d)$ has a between point $z \in X$, so that there are $x,y \in X \backslash \{z\}$ with $d(x,y) = d(x,z) + d(z,y)$. We want to show that $X$ is geodesic, but by property (\ref{main}.vi), it suffices to show that there is a geodesic segment from $x$ to $y$. We may also assume, without loss of generality, that $d(x,y)=1$. Our goal, then, is to construct an isometric embedding $\gamma \colon [0,1] \rightarrow X$ with $\gamma(0)=x$ and $\gamma(1)=y$.

Let $\delta = d(x,z)$ so that $0 < \delta < 1$. For each $n \in \N$, we decompose $[0,1]$ into non-overlapping closed intervals $I_{a_1, \ldots, a_n}$, indexed by $(a_1,\ldots,a_n) \in \{0,1\}^n$ in the following way:
\begin{enumerate}[(i)]
\item $I_0 = [0,\delta]$ and $I_1 = [\delta, 1]$;
\item $I_{a_1, \ldots, a_n,0} \subset I_{a_1, \ldots, a_n}$ is of length $\delta \cdot |I_{a_1, \ldots, a_n}|$ and shares a left endpoint with $I_{a_1, \ldots, a_n}$;
\item $I_{a_1, \ldots, a_n,1} \subset I_{a_1, \ldots, a_n}$ is of length $(1-\delta) \cdot |I_{a_1, \ldots, a_n}|$ and shares a right endpoint with $I_{a_1, \ldots, a_n}$.
\end{enumerate}
Let $E_n$ be the set of endpoints of the intervals $I_{a_1,\ldots,a_n}$ as $(a_1,\ldots,a_n)$ ranges through $\{0,1\}^n$, and note that $\cup_n E_n$ is dense in $[0,1]$. Define $\gamma \colon \cup_n E_n \rightarrow X$ inductively as follows:
\begin{enumerate}[(i)]
\item $\gamma(0) = x$, $\gamma(\delta) = z$, and $\gamma(1)=y$;
\item If $s<t<u$ with $s,u$ the endpoints of $I_{a_1, \ldots, a_n}$ and $t$ the right endpoint of $I_{a_1, \ldots, a_n,0}$, then $\gamma(t)$ is a point for which
$$d(\gamma(s),\gamma(t)) = \delta \cdot d(\gamma(s),\gamma(u)) \hspace{0.2cm} \text{and} \hspace{0.2cm} d(\gamma(t),\gamma(u)) = (1-\delta) \cdot d(\gamma(s),\gamma(u)).$$
\end{enumerate}
To see that this definition makes sense, observe that property (\ref{main}.vi) implies that for every two points $x',y' \in X$, there is a point $z' \in X$ with 
$$d(x',z') = \delta \cdot d(x',y') \hspace{0.3cm} \text{ and } \hspace{0.3cm} d(z',y') = (1-\delta)\cdot d(x',y').$$
It is not difficult to verify that $\gamma$, thus defined, is an isometry on $\cup_n E_n$. As $(X,d)$ is complete (cf. Remark \ref{completeness}), we can therefore extend $\gamma$ to an isometry on $[0,1]$. This proves the first part of the proposition.

For the second part, assume that $(X,d)$ has no between-point. We may assume that $X$ has at least two points, for otherwise the desired conclusion is vacuously true. Lemma \ref{existsp} guarantees that there is $p > 1$ for which $(X,d^p)$ is a metric space with at least one between-point. Notice, of course, that $(X,d^p)$ still satisfies properties (\ref{main}.i), (\ref{main}.ii), and (\ref{main}.iii). By the first part of the proposition, then, we can conclude that $(X,d^p)$ is geodesic.

The final statement in the proposition follows immediately from the previous two parts.
\end{proof}

\begin{proof}[\textbf{Proof of Theorem \ref{main}}]
In light of Proposition \ref{prop}, it suffices to prove the following statement. Any geodesic metric space that satisfies properties (\ref{main}.i), (\ref{main}.ii), and (\ref{main}.iii) is isometric to $\E^n$ for some $n \in \N \cup \{0\}$. The case when $X$ has one point is clear, so we may assume that $X$ has at least two points.

For such $(X,d)$, Characterization \ref{LD} implies that $X$ is isometric to a sub-Finsler Carnot group. Moreover, $X$ has the stronger property that $\Isom(X)$ is transitive on pairs of points at distance $1$. In particular, the group $\Isom(X)_e$ of isometries fixing the identity element acts transitively on the unit sphere $S(e,1)$.

The isometries of a sub-Finsler Carnot group are known to be affine \cite{LDO12}. In particular, those that fix the identity element are group isomorphisms that preserve the horizontal stratum $V_1$, and their horizontal differentials are isometries. Thus, every point of the form $\exp(v)$ with $v \in V_1$ can be mapped, under an isometry of $X$, only to points of the form $\exp(w)$ with $w \in V_1$ and $|v|=|w|$. We deduce, then, that $S(e,1)=\exp(V_1)\cap S(e,1)$, which is possible only if the Carnot group is $\exp(V_1)$. In other words, the Carnot group has only one stratum, which means that it is abelian. The only abelian Carnot groups are $\R^n$, $n \in \N$. Hence, $X$ is isometric to some (finite-dimensional) normed space $(\R^n,|\cdot|)$.

Of course, $(\R^n,|\cdot|)$ also has the stronger property that its group of isometries fixing 0 is linear and acts transitively on the $|\cdot|$-unit sphere. It is a classical fact that such norms necessarily come from a scalar product (cf. \cite[Exercise 1.2.24]{BBI01}). Thus, $X$ is isometric to a Euclidean space.
\end{proof}

\section{An example}\label{ex:Cantor}

We 
recall here an example of a 
metric space
that
  satisfies (\ref{LD}.i), (\ref{LD}.iii), and (\ref{LD}.iv) 
 but fails to be connected 
 and locally connected.
 We thank
Gareth Speight 
 for reminding us of this example.
 
  Let $X$ be the set of all double-sided sequences
   $(x_{n})_{n\in\Z}$ 
  such that $x_{n}\in \{0,1\}$ and 
  that are eventually $0$, i.e.,
  there is an integer $N$ such that $x_{n}=0$ for all $n>N$.
Define a metric on $X$ by
$$d((x_{n}),(y_{n})):=2^{\max\{n : x_{n} \neq y_{n}\}}.$$
 It is straightforward to check that the metric space $(X,d)$ is 
 locally compact, isometrically homogeneous, and self similar (a dilation is provided by the shift map).	

Note, however, that $(X,d)$ does not admit dilations of {\em every} factor, in the sense of property (\ref{main}.iii). Indeed, otherwise Lemma \ref{HDim1} would 
 imply that the topological dimension is at least $1$, a contradiction.

\end{document}